\documentclass[a4paper,12pt]{article}

\usepackage[left=2cm,right=2cm, top=2cm,bottom=2cm,bindingoffset=0cm]{geometry}

\usepackage{verbatim}
\usepackage{amsmath}
\usepackage{amsthm}
\usepackage{amssymb}
\usepackage{delarray}
\usepackage{hyperref}
\usepackage{mathrsfs}

\newcommand{\al}{\alpha}
\newcommand{\be}{\beta}
\newcommand{\ga}{\gamma}
\newcommand{\de}{\delta}
\newcommand{\la}{\lambda}
\newcommand{\om}{\omega}

\newcommand{\vv}{\varphi}
\newcommand{\iy}{\infty}

\theoremstyle{plain}

\newtheorem{thm}{Theorem}
\newtheorem{lem}{Lemma}
\newtheorem{cor}{Corollary}

\theoremstyle{definition}

\theoremstyle{remark}

\DeclareMathOperator*{\Res}{Res}

 \allowdisplaybreaks

\begin{document}

\begin{center}
{\large\bf Matrix Sturm-Liouville equation with a Bessel-type singularity on a finite interval
}
\\[0.2cm]
{\bf Natalia Bondarenko} \\[0.2cm]
\end{center}

\vspace{0.5cm}

{\bf Abstract.} The matrix Sturm-Liouville equation on a finite interval with a Bessel-type singularity in the end of the interval
is studied. Special fundamental systems of solutions for this equation are constructed: analytic Bessel-type solutions with the prescribed behavior at the singular point and Birkhoff-type solutions with the known asymptotics for large values of the spectral parameter. The asymptotic formulas for Stokes multipliers, connecting these two fundamental systems of solutions, are derived. We also set boundary conditions and obtain asymptotic formulas for the spectral data (the eigenvalues and the weight matrices) of the boundary value problem. Our results will be useful in the theory of direct and inverse spectral problems.

\medskip


\medskip


\vspace{1cm}

{\large \bf 1. Introduction}

\medskip

Consider the matrix Sturm-Liouville equation on a finite interval with a Bessel-type singularity
in the end of the interval
\begin{equation} \label{eqv}
 	-Y'' + \left( \frac{\omega}{x^2} + Q(x) \right) Y = \lambda Y, \quad x \in (0, T). 
\end{equation}
Here $Y = [y_k(x)]_{k = 1}^m$ is a vector function, $\lambda$ is the spectral parameter,
$Q(x)$ and $\omega$ are $m \times m$ matrices. 

We assume that the matrix $\omega$ is diagonal, i.e.
$\omega = \mbox{diag}\{\omega_1, \omega_2, \dots, \omega_m \}$, $\omega_q \in \mathbb{R}$, $q = \overline{1, m}$.
If $\om$ is an arbitrary Hermitian matrix, one can apply the standard unitary transform, in order to fulfill this condition.
For definiteness, let
$\omega_q = \nu_q^2 - \frac{1}{4}$, $\nu_1 \ge \nu_2 \ge \dots \ge \nu_m > 0$, $\nu_q \notin \mathbb{N}$, $q = \overline{1, m}$.
Let the matrix function $x^{1 - 2\nu_1} Q(x)$ be integrable on $(0, T)$. 

In this paper, we construct special fundamental systems of solutions (FSS) for equation \eqref{eqv} and study asymptotic behavior of the spectral data. Our results will be useful in the theory of direct and inverse spectral problems 
for systems with Bessel-type singularities.

Differential equations with regular singularities are studied intensively in recent years. 
The direct and inverse spectral problems for scalar analogue of equation \eqref{eqv} ($m = 1$) have been investigated in \cite{FY05-1} and \cite{FY05-2}, 
respectively, even for the more general case of the singularities in the both ends of the interval. 
These problems arose from the analysis of the differential equations with turning points. Such equations have applications 
in electronics and geophysics (see, for example, \cite{FY99, And97, LU81}). The case of singularities inside the interval 
has also been investigated \cite{Yur97, Fed13}. In papers \cite{GYS15, GY15} the Dirac system with regular singularity was studied. The papers \cite{Ign15, Yur15} are devoted to inverse problems for differential operators with singularities on geometrical graphs. 

The approach in the mentioned papers is based on the special FSS, constructed in \cite{Yur-depon, Yur92} for higher-order differential equations with the regular singularity. The first FSS is formed by Bessel-type solutions, which are analytic in $\la$
and have the prescribed asymptotic behavior at the singular point. The second FSS consists of Birkhoff-type solutions 
with the known behavior for large values of the spectral parameter. The two FSS are connected by
the Stokes multipliers. The asymptotics for the Stokes multipliers play a crucial role in the spectral theory for
operators with singularities, since they give an opportunity to analyze the behavior of solutions and spectral characteristics for these operators.

The matrix equation \eqref{eqv} with the Bessel-type singularity have been studied in \cite{AM60},
where inverse scattering problem, related to quantum mechanics, was solved.
However, the method of \cite{AM60}, based on special non-integral 
transforms, removing singularities, works only in some particular cases. In paper \cite{Bond15}, 
the inverse spectral problem for equation \eqref{eqv} on the finite interval was studied in the more general case.
The uniqueness theorem was proved.

In the present paper, we continue the research, started in \cite{Bond15}. We provide the construction of the 
FSS for equation \eqref{eqv} and derive asymptotic formulas for the Stokes multipliers, which appeared in the short note 
\cite{Bond15} without proofs. Further we set the boundary value problem for equation \eqref{eqv} and obtain 
the asymptotics for its spectral data, consisting of the eigenvalues and so-called weight matrices.
The weight matrices equal to the residues of the Weyl matrix, they generalize the notion of the weight numbers for the scalar Sturm-Liouville equation (see \cite{FY01}). The spectral data, considered in this paper, are natural spectral characteristics,
by which a matrix Sturm-Liouville operator can be recovered. The inverse problem by the spectral data
for the matrix Sturm-Liouville equation without singularities was studied, for example, in \cite{Bond11} 
(see also references therein). The asymptotics for the eigenvalues and for the weight matrices can further be used 
for the constructive solution of the inverse problem for equation \eqref{eqv}. We note that the asymptotic
analysis of the spectral data of the considered problem involves difficulties, related with the 
asymptotic closeness of large eigenvalues. Therefore one can not calculate asymptotics of the weight matrices separately,
but can only analyze the sums of the residues, corresponding to some contours. Such approach allowed to solve inverse problems for matrix Sturm-Liouville operators in \cite{Bond11, Bond15-2}.

The paper is organized as follows. For the convenience of the reader, we provide the standard results on the 
FSS for the scalar Sturm-Liouville equation with the Bessel singularity in Section 2. In Section~3, we construct the Bessel-type and the 
Birkhoff-type solutions for the matrix equation \eqref{eqv}. In Section~4, we establish the connection between
two types of the solutions and obtain asymptotics for the Stokes multipliers. 
In Section~5, the boundary conditions are set for equation \eqref{eqv}, and the notion of the spectral data is introduced.
We derive asymptotic formulas for the eigenvalues and the weight matrices in Sections~6 and~7, respectively.	

Let us introduce the notation. Denote $\rho := \sqrt \la$, $\mbox{Re}\,\rho \ge 0$, $\tau := \mbox{Im}\,\rho$. 
Denote the unit $m \times m$ matrix by $I$, its $q$-th column by $e_q$ and its $q$-th row by $e_q^*$. 
We use the following matrix norm: $\| A \| = \max\limits_{j,k = \overline{1, m}} |a_{jk}|$, $A = [a_{jk}]_{j, k = 1}^m$.
The Wronskian $\langle Z, Y \rangle := Z Y' - Z' Y$ can be used for both scalar and matrix functions. 
In asymptotics, we use the symbol $[I]_{\beta} := I + O(\rho^{-\beta})$, where $\beta := \min\{ 1, 2 \nu_1 \}$. 
 

\bigskip

{\large \bf 2. Scalar case}

\bigskip

In this section, we briefly recall the construction of the FSS for the scalar equation
\begin{equation} \label{eqvs}
-y'' + \frac{\om}{x^2} y = \la y,
\end{equation}
where $\omega = \nu^2 - \frac{1}{4}$, $\nu > 0$, $\nu \notin \mathbb N$. We adapt the results from \cite{Yur-depon, Yur92}, where they were obtained for 
the higher-order differential operators.

Let $\la = 1$. Then equation \eqref{eqvs} has a fundamental system of Bessel solutions
\begin{equation} \label{defc0}
c_j(x) = x^{\mu_j} \sum_{k = 0}^{\iy} c_{jk} x^{2k}, \quad j = 1, 2,
\end{equation}
where 
$$
\mu_1 = \frac{1}{2} - \nu, \quad \mu_2 = \frac{1}{2} + \nu,
$$
$$
c_{jk} = (-1)^k c_{j0} \left( \prod_{s = 1}^k ((2s + \mu_j) (2 s + \mu_j - 1) + \omega) \right)^{-1}.
$$
The constants $c_{j0}$ can be chosen arbitrary. Fix them in such a way, that $c_{10} c_{20} = (2 \nu)^{-1}$.
Then 
\begin{equation} \label{wronc}
\langle c_1, c_2 \rangle = 1.
\end{equation}

Denote by $\Pi$ the complex $x$-plane with the cut $x \le 0$. The solutions $c_j(x)$ are analytic in~$\Pi$.

Equation \eqref{eqvs} with $\la = 1$ also has Jost solutions
$$
e_1(x) = e^{ix} - \int_x^{\iy} \sin(x - t) \frac{\omega}{t^2} e_1(t) \, dt, \quad \mbox{Im}\, x > 0,  
$$ 
\begin{equation} \label{defe2}
e_2(x) = e_1(-x), \quad \mbox{Im}\, x < 0.
\end{equation}
Here the integral is taken along the ray $\{ t \colon \arg t = \arg x, |t| \in (|x|, +\iy) \}$. The functions 
$e_1(x)$ and $e_2(x)$ are analytic in the upper and the lower half-plane, respectively. They can be represented 
in the form
\begin{equation} \label{ec0}
e_k(x) = \beta_{k1}^0 c_1(x) + \beta_{k2}^0 c_2(x), \quad k = 1, 2,
\end{equation}
with some constants $\beta_{kj}^0$. Relation \eqref{ec0} gives an analytic continuation of $e_k(x)$ into the whole
$\Pi$. One can show that the following asymptotic relations hold for $\nu = 0, 1$:
$$
e_1^{(\nu)}(x) = i^{\nu} \exp(ix) (1 + O(x^{-1})), \quad e_2^{(\nu)}(x) = (-i)^{\nu} \exp(-ix) (1 + O(x^{-1})), \quad x \in \Pi, \quad |x| \to \iy.
$$
Hence 
\begin{equation} \label{wrone}
\langle e_1, e_2 \rangle = - 2 i.
\end{equation}

It follows from \eqref{defe2} and \eqref{ec0}, that 
\begin{equation} \label{relbeta}
\beta_{2j}^0 = \exp(i \pi \mu_j) \beta_{1j}^0, \quad j = 1, 2.  
\end{equation}
By virtue of \eqref{wronc}, \eqref{wrone} and \eqref{ec0}, $\det [\beta_{kj}^0]_{k, j = 1, 2} = - 2 i$.
Consequently, 
$$
   \beta_{11}^0 \beta_{12}^0 = \frac{1}{4 i \sin \nu}.
$$

Now consider equation \eqref{eqvs} for $\la = \rho^2$, $\mbox{Re}\, \rho \ge 0$, $x > 0$.
Obviously, the functions
\begin{equation} \label{defce}
   c_j(x, \la) = \rho^{-\mu_j} c_j(\rho x), \quad j = 1, 2, \quad e_k(x, \rho) = e_k(\rho x), \quad k = 1, 2,
\end{equation} 
form FSS for it. The solutions $c_j(x, \la)$ are entire in $\la$, while $e_k(x, \rho)$ are analytic
in $\rho$ for $\mbox{Re}\,\rho > 0$, continuous for $\mbox{Re}\,\rho \ge 0$, $|\rho| \ge \rho^*$ and satisfy the estimates
\begin{equation} \label{asympte1}
 	|e_k^{(\nu)}(x, \rho) ((-1)^{k-1} i \rho)^{\nu} \exp((-1)^{k - 1} i \rho x) - 1| \le \frac{M_0}{|\rho|x}, \quad |\rho| x \ge 1, \quad \nu = 0, 1, \quad k = 1, 2.
\end{equation}
with some constant $M_0$.

The relations \eqref{ec0} yield
\begin{equation} \label{ecb}
 	e_k(x, \rho) = \beta_{k1}^0 \rho^{\mu_1} c_1(x, \la) + \beta_{k2}^0 \rho^{\mu_2} c_2(x, \la), \quad k = 1, 2.
\end{equation}
Consequently, in view of \eqref{defc0} and \eqref{defce}, we have
\begin{equation} \label{asympte2}
e_k(x, \rho) = O((\rho x)^{\mu_1}), \quad |\rho|x < 1, \, k = 1, 2.
\end{equation}

\bigskip

{\bf \large 3. Bessel-type and Birkhoff-type solutions for the matrix equation}

\bigskip

The matrix equation \eqref{eqv} with $Q(x) \equiv 0$ splits into $m$ scalar equations, and therefore it
has matrix solutions $C_j(x, \la) = \mbox{diag}\, \{ c_{jq}(x, \la) \}_{q = 1}^m$, 
$E_j(x, \rho) = \mbox{diag}\, \{ e_{jq}(x, \rho) \}_{q = 1}^m$, $\rho = \sqrt \la$, $\mbox{Re}\, \rho \ge 0$, $j = 1, 2$, 
where $c_{jq}(x, \la)$ and $e_{jq}(x, \rho)$
are constructed like $c_j(x, \la)$ and $e_j(x, \rho)$ in Section~2, respectively, with $\nu_q$ instead of $\nu$. 
In particular, $\mu_{1q} = \frac{1}{2} - \nu_q$,
$\mu_{2q} = \frac{1}{2} + \nu_q$, $q = \overline{1, m}$.

Let $S_j(x, \la)$ and $S_j^*(x, \la)$, $j = 1, 2$, be the matrix solutions of the following integral equations
\begin{gather} 
\label{inteqS}
 	S_j(x, \la) = C_j(x, \la) - \int_0^x G(x, t, \la) Q(t) S_j(t, \la) \, dt, \\
\label{inteqS*}
 	S_j^*(x, \la) = C_j(x, \la) - \int_0^x S_j^*(t, \la) Q(t) G(x, t, \la) \, dt.
\end{gather}
where $G(x, t, \la) = C_2(x, \la) C_1(t, \la) - C_1(x, \la) C_2(t, \la)$.

Along with equation \eqref{eqv}, consider the following equation
\begin{equation} \label{eqv*}
 	-Z'' + Z \left( \frac{\omega}{x^2} + Q(x) \right) = \lambda Z, \quad x \in (0, T), 
\end{equation}
where $Z = Z(x)$ is a row vector. 

The matrix-functions $S_j(x, \la)$ and $S_j^*(x, \la)$, $j = 1, 2$, are called {\it Bessel-type solutions}
for equations \eqref{eqv} and \eqref{eqv*}, respectively.
They have the following properties \cite{Bond15}:

\begin{enumerate}
\item The columns of the matrices $S_j(x, \la)$, $j = 1,2$, and the rows of the matrices $S_j^*(x, \la)$, $j = 1, 2$,
form FSS for equations \eqref{eqv} and \eqref{eqv*}.
\item For each fixed $x \in (0, T)$, the matrix functions 
$S_j^{(\nu)}(x, \lambda)$ � $S_j^{*(\nu)}(x, \lambda)$ are entire in the $\la$-plane.
\item The following asymptotic formulas are valid as $x \to 0$, $j = 1, 2$, $q = \overline{1, m}$:
\begin{equation} \label{asymptS1}
 	S_j(x, \lambda) e_q = O(x^{\mu_{jq}}), \quad e^*_q S_j^*(x, \lambda) = O(x^{\mu_{jq}}), 
\end{equation}
\begin{equation} \label{asymptS2}
 	x^{-\mu_{jq}} (S_j(x, \la) - C_j(x, \la)) e_q = o(x^{2\nu_1}), \quad
 	x^{-\mu_{jq}} e_q^* (S_j^*(x, \la) - C_j(x, \la)) = o(x^{2 \nu_1}),
\end{equation}
\item The following relations hold
\begin{equation} \label{wronS}
 	\langle S^*_j(x, \la), S_k(x, \la) \rangle = (-1)^{j-1} \de_{jk} I, \quad j, k = 1, 2,
\end{equation}
where $\de_{jk}$ is the Kronecker delta.
\end{enumerate}

In the next theorem, we construct the so-called {\it Birkhoff-type solutions} with the prescribed behavior as $|\rho| \to \iy$.

\begin{thm}\label{thm:Y}
Let $\Omega_{0} = \{ \rho \colon \arg \rho \in (0, \pi/2), \, |\rho| > \rho^* \}$, 
$\Omega_{-1} = \{ \rho \colon \arg \rho \in (-\pi/2, 0), \, |\rho| > \rho^* \}$. 
In each fixed sector $\Omega_{k_0}$ ($k_0 \in \{0, 1 \}$) there exist matrix functions $Y_k(x, \rho)$, $k = 1,2$, 
whose columns form FSS for equation \eqref{eqv} and satisfy the following conditions:

($i_1$) For each fixed $x \in (0, T)$, the matrix functions $Y_k^{(\nu)}(x, \rho)$ are analytic in $\Omega_{k_0}$
and continuous in $\overline{ \Omega_{k_0}}$ for sufficiently large $\rho^*$.

($i_2$) For $x \in (0, T]$, $\rho \in \overline{\Omega_{k_0}}$, $|\rho x| \ge 1$, $\nu = 0, 1$, the following asymptotic 
formulas are valid:
\begin{equation} \label{asymptY}
 Y_1^{(\nu)}(x, \rho) = (i \rho)^{\nu} \exp(i \rho x) [I]_{\beta}, \quad
 Y_2^{(\nu)}(x, \rho) = (- i \rho)^{\nu} \exp(- i \rho x) [I]_{\beta}, \quad |\rho| \to \iy
\end{equation}
where $[I]_{\beta} = \left( I + O(\rho^{-\beta}) \right)$, $\beta = \min \{ 1, 2 \nu_1 \}$.
\end{thm}

\begin{proof}
For definiteness, consider the sector $\Omega_0$. The arguments for $\Omega_{-1}$ are similar.
Consider the integral equations
\begin{multline} \label{intY1}
Y_1(x, \rho) = E_1(x, \rho) + \frac{1}{2 i \rho} \int_0^x E_1(x, \rho) E_2(t, \rho) Q(t) Y_1(t, \rho) \, dt
\\ + \frac{1}{2 i \rho} \int_x^T E_2(x, \rho) E_1(t, \rho) Q(t) Y_1(t, \rho) \, dt, 
\end{multline}
\begin{equation} \label{intY2}
   Y_2(x, \rho) = E_2(x, \rho) + \frac{1}{2 i \rho} \int_0^x (E_1(x, \rho) E_2(t, \rho) - E_2(x, \rho) E_1(t, \rho)) Q(t) Y_2(t, \rho) \, dt.
\end{equation} 
Obviously, their solutions $Y_k(x, \rho)$, $k = 1, 2$, formally satisfy \eqref{eqv}. Denote for $k = 1, 2$
\begin{equation} \label{uk1}
    u_k(x, \rho) = \frac{Y_k(x, \rho)}{(\rho x)^{\mu_{11}}}, \quad u_k^0(x, \rho) = \frac{E_k(x, \rho)}{(\rho x)^{\mu_{11}}} \quad |\rho x| < 1, 
\end{equation}
\begin{equation*} 
    u_k(x, \rho) = \frac{Y_k(x, \rho)}{\exp((-1)^{k-1} i \rho x)}, \quad  u_k^0(x, \rho) = \frac{E_k(x, \rho)}{\exp((-1)^{k-1} i \rho x)}, \quad |\rho x| \ge 1.
\end{equation*}
Note that by virtue of \eqref{asympte1} and \eqref{asympte2}, $ u_k^0(x, \rho) = O(1)$, $x \in (0, T]$, $\rho \in \Omega_0$.

The equations \eqref{intY1} and \eqref{intY2} can be rewritten in the form
\begin{equation} \label{intu}
   u_k(x, \rho) = u_k^0(x, \rho) + \int_0^T \mathscr{K}_k(x, t, \rho) u_k(t, \rho) \, dt, \quad k = 1, 2.
\end{equation}
The kernel $\mathscr{K}_1(x, t, \rho)$ takes the form
$$
	|\rho x| < 1 \colon \quad
   \mathscr{K}_1(x, t, \rho) = \begin{cases}
      \frac{1}{2 i \rho} u_1^0(x, \rho) u_2^0(t, \rho) (\rho t)^{2 \mu_{11}} Q(t), \quad 0 < t < x, \\
      \frac{1}{2 i \rho} u_2^0(x, \rho) u_1^0(t, \rho) (\rho t)^{2 \mu_{11}} Q(t), \quad x < t < \frac{1}{|\rho|}, \\
      \frac{1}{2 i \rho} u_2^0(x, \rho) u_1^0(t, \rho) \exp(2 i \rho t) Q(t), \quad \frac{1}{|\rho|} < t < T,
   \end{cases}
$$
$$
	|\rho x| \ge 1 \colon \quad
   \mathscr{K}_1(x, t, \rho) = \begin{cases}
      \frac{1}{2 i \rho} u_1^0(x, \rho) u_2^0(t, \rho) (\rho t)^{2 \mu_{11}} Q(t), \quad 0 < t < \frac{1}{|\rho|}, \\
      \frac{1}{2 i \rho} u_2^0(x, \rho) u_1^0(t, \rho) \exp(2 i \rho t) Q(t), \quad \frac{1}{|\rho|} < t < x, \\
      \frac{1}{2 i \rho} u_2^0(x, \rho) u_1^0(t, \rho) \exp(2 i \rho t) Q(t), \quad x < t < T.
   \end{cases}
$$
Similar representations can be written for $\mathscr{K}_2(x, t, \rho)$.
Clearly, 
\begin{equation} \label{estK}
 	\int_0^T \| \mathscr{K}_k(x, t, \rho) \| \, dt \le \frac{C}{|\rho|^{-\beta}}, \quad x \in (0, T], \quad \rho \in \overline{\Omega_0}, \quad k = 1,2.
\end{equation}
Therefore the solutions of \eqref{intu} can be found by the method of successive approximations for $|\rho| > \rho^*$, 
if $\rho^*$ is sufficiently large. The solutions $u_k(x, \rho)$ are analytic in $\Omega_0$,
continuous in $\overline{ \Omega_0}$ and satisfy the estimates
$$
  	\| u_k(x, \rho) \| \le C, \quad \| u_k(x, \rho) - u_k^0(x, \rho) \| \le \frac{C}{|\rho|^{-\beta}}, \quad x \in (0, T], \, \rho \in \overline{\Omega_0}.
$$
Consequently, the conditions ($i_1$) and ($i_2$) are fulfilled for the solutions $Y_k(x, \rho)$, $k = 1, 2$, of \eqref{intY1}, \eqref{intY2}.
\end{proof}

\bigskip

{\large \bf 4. Stokes multipliers}

\bigskip

In this section, we investigate the connection between the Bessel-type and the Birkhoff-type solutions.
Clearly, there exist matrix coefficients $B_{kj}(\rho)$, called the {\it Stokes multipliers}, such that
\begin{equation} \label{YSB}
Y_k(x, \rho) = S_1(x, \la) B_{k1}(\rho) + S_2(x, \la) B_{k2}(\rho), \quad k = 1, 2.
\end{equation}
For the case $Q(x) \equiv 0$, we have
\begin{equation} \label{YSB0}
E_k(x, \rho) = C_1(x, \la) B_{k1}^0(\rho) + C_2(x, \la) B_{k2}^0(\rho)
\end{equation}
Recall that the matrices $C_j(x, \la)$ and $E_k(x, \rho)$, $j, k = 1, 2$, are diagonal. In view of \eqref{ecb},
$$
  	B_{kj}^0(\rho) = D_j(\rho) B_{kj}^0, \quad D_j(\rho) = \mbox{diag}\{ \rho^{\mu_{jq}} \}_{q = 1}^m, 
  	\quad B_{kj}^0 = \mbox{diag}\{ \beta_{kjq}^0 \}_{q = 1}^m.
$$
The following theorem presents asymptotic formulas for the Stokes multipliers $B_{kj}(\rho)$.
These asymptotics play a crucial role in the analysis of direct and inverse problems for equation \eqref{eqv}.
Further we use these relations in order to derive asymptotic formulas for the solutions $S_j(x, \la)$ as $|\la| \to \iy$, 
and for the spectral characteristics of equation \eqref{eqv}.

\begin{thm} \label{thm:Stokes}
The following relations hold
\begin{equation} \label{asymptB}
  	B_{kj}(\rho) = D_j(\rho) B_{kj}^0 [I]_{\beta}, \quad \mbox{Re}\,\rho \ge 0, \quad |\rho| \to \iy, \quad k, j = 1, 2.
\end{equation}
\end{thm}

In order to prove Theorem~\ref{thm:Stokes}, we need the following auxiliary lemma.

\begin{lem} \label{lem:F}
The vector functions
\begin{equation} \label{defF1}
 	F_{1q}(x, \rho) = e_q^* (E_1(x, \rho) - Y_1(x, \rho)) (\rho x)^{-\mu_{1q}}, 
\end{equation}
\begin{equation} \label{defF2}
 	F_{2q}(x, \rho) = (F_{1q}(x, \rho) - c_{10q}^{-1} \hat c_{1q}(x, \la) F_{1q}(0, \rho)) (\rho x)^{\mu_{1q} - \mu_{2q}}, \quad q = \overline{1, m},
\end{equation}
are continuous at $x = 0$ and
\begin{equation} \label{asymptF}
 	F_{kq}(0, \rho) = O(\rho^{-\beta}), \quad \rho \in \overline{\Omega_0}, \quad |\rho| \to \iy, \quad k = 1, 2, \, q = \overline{1, m}.
\end{equation}
Here $\hat c_{jq}(x, \la) := x^{-\mu_{1q}} c_{jq}(x, \la)$, and $c_{10q}$ is the coefficient in the series \eqref{defc0} for $c_{1q}(x)$.
\end{lem}

\begin{proof}
Fix $q = \overline{1, m}$.
It follows from \eqref{intu} and \eqref{defF1}, that
\begin{equation} \label{smF3}
 	- F_{1q}(x, \rho) = e_q^* (\rho x)^{\mu_{11} - \mu_{1q}} \int_0^T \mathscr{K}_1(x, t, \rho) u_1(t, \rho) \, dt,
\end{equation}
where $\mathscr{K}_1$ and $u_1$ were defined in the proof of Theorem~\ref{thm:Y}. 
We are interested in the case $|\rho x| < 1$, $x \to 0$, $\rho \in \overline{\Omega_0}$.
The relations \eqref{uk1} and \eqref{ecb}
imply
\begin{equation} \label{smF}
  	e_q^* (\rho x)^{\mu_{11} - \mu_{1q}} u_k^0(x, \rho) = (\rho x)^{-\mu_{1q}} e_{kq}(x, \rho) e_q^* =
  	\bigl[ \beta_{k1q}^0 \hat c_{1q}(x, \la) + \beta_{k2q}^0 (\rho x)^{\mu_{2q} - \mu_{1q}} \hat c_{2q}(x, \la) \bigr] e_q^*, 
\end{equation}
for $k = 1, 2$. 
This expression tends to $\beta_{k1q}^0 c_{10q}$, as $x \to 0$.
Consequently, the kernel $e_q^* (\rho x)^{\mu_{11} - \mu_{1q}} \mathscr{K}_1(x, t, \rho)$
is continuous at $x = 0$ and satisfies the estimate, similar to \eqref{estK}, uniformly by $x$ in the neighborhood of $0$. 
Hence \eqref{asymptF} is valid for $F_{1q}$.

Let us prove \eqref{asymptF} for $F_{2q}$. Substitute the representation for $\mathscr{K}_1(x, t, \rho)$ and \eqref{smF}
into \eqref{smF3}:
\begin{multline*}
  - F_{1q}(x, \rho) = \frac{1}{2 i \rho} \int_0^x \bigl[ \beta_{11q}^0 \hat c_{1q}(x, \la) + (\rho x)^{\mu_{2q} - \mu_{1q}}
  \beta_{12q}^0 \hat c_{2q}(x, \la) \bigr] (\rho t)^{2 \mu_{11}} e_q^* u_2^0(t, \rho) Q(t) u_1(t, \rho) \, dt \\
  + \frac{1}{2 i \rho} \int_x^{\frac{1}{|\rho|}} \bigl[ \beta_{21q}^0 \hat c_{1q}(x, \la) + (\rho x)^{\mu_{2q} - \mu_{1q}}
  \beta_{22q}^0 \hat c_{2q}(x, \la) \bigr] (\rho t)^{2 \mu_{11} } e_q^* u_1^0(t, \rho) Q(t) u_1(t, \rho) \, dt \\
  + \frac{1}{2 i \rho} \int_{\frac{1}{|\rho|}}^T \bigl[ \beta_{21q}^0 \hat c_{1q}(x, \la) + (\rho x)^{\mu_{2q} - \mu_{1q}}
  \beta_{22q}^0 \hat c_{2q}(x, \la) \bigr] \exp(2 i \rho t) e_q^* u_1^0(t, \rho) Q(t) u_1(t, \rho) \, dt, \quad |\rho x| < 1,
\end{multline*}
\begin{multline*}
- c_{10q}^{-1} \hat c_{1q}(x, \la) F_{1q}(0, \rho) = \frac{1}{2 i \rho} \int_0^{\frac{1}{|\rho|}} \beta_{21q}^0 \hat c_{1q}(x, \la) (\rho t)^{2 \mu_{11}} 
e_q^* u_1^0(t, \rho) Q(t) u_1(t, \rho) \, dt \\
+ \frac{1}{2 i \rho} \int_{\frac{1}{|\rho|}}^T \beta_{21q}^0 \hat c_{1q}(x, \la) \exp(2 i \rho t) e_q^* u_1^0(t, \rho) Q(t) u_1(t, \rho) \, dt.   
\end{multline*}
Now substitute these two relations into \eqref{defF2}. Clearly, the terms with $\hat c_{2q}(x, \la)$ after multiplication
by $(\rho x)^{\mu_{1q} - \mu_{2q}}$ are continuous at $x = 0$ and fulfill the estimate $O(\rho^{-\beta})$ uniformly with respect to $x$. The terms with $\hat c_{1q}(x, \la)$ 
vanish in the integrals by $(x, \frac{1}{|\rho|})$ and $(\frac{1}{|\rho|}, T)$. Under the integral by $(0, x)$, we have
$$
\hat c_{1q}(x, \la) e_q^* (\beta_{11q}^0 u_2^0(t, \rho) - \beta_{21q}^0 u_1^0(t, \rho)) =
\hat c_{1q}(x, \la) (\beta_{11q}^0 e_2(t, \rho) - \beta_{21q}^0 e_1(t, \rho)) (\rho t)^{-\mu_{11}} e_q^*. 
$$
Using \eqref{ecb}, one can easily derive
$$
  \beta_{11q}^0 e_2(t, \rho) - \beta_{21q}^0 e_1(t, \rho) = (\beta_{11q}^0 \beta_{22q}^0 - \beta_{21q}^0 \beta_{12q}^0) \hat c_2(t, \la) (\rho t)^{\mu_{2q}}.
$$
So we have the following integral
\begin{equation} \label{smint}
   \frac{1}{\rho} \int_0^x \hat c_{1q}(x, \la) \hat c_2(t, \la) (\rho x)^{\mu_{1q} - \mu_{2q}} (\rho t)^{\mu_{2q} - \mu_{11}} (\rho t)^{2 \mu_{11}} e_q^*
   Q(t) u_1(t, \rho) \, dt.
\end{equation}
Since $t < x$ and $|\rho x| < 1$, we can estimate 
$$
   |(\rho x)^{\mu_{1q} - \mu_{2q}} (\rho t)^{\mu_{2q} - \mu_{11}}| \le |\rho x|^{\mu_{1q} - \mu_{11}} \le 1.
$$
Consequently, the integral \eqref{smint} is continuous at $x = 0$ and fulfill the estimate $O(\rho^{-\beta})$.
Thus, the lemma is proved for $F_{2q}$. 
\end{proof}

\begin{proof}[Proof of Theorem~\ref{thm:Stokes}]
Let us prove the relations \eqref{asymptB} for $\rho \in \overline{\Omega_0}$. The case $\rho \in \overline{\Omega_{-1}}$ is analogous. 
The relations \eqref{YSB}, \eqref{YSB0} and \eqref{defF1} yield
\begin{multline} \label{smF1}
 	F_{1q}(x, \rho) = e_q^* \sum_{j = 1}^2 (\rho x)^{-\mu_{1q}} C_j(x, \la) (B_{1j}^0(\rho) - B_{1j}(\rho)) + 
 	e_q^* \sum_{j = 1}^2 (\rho x)^{-\mu_{1q}} (S_j(x, \la) - C_j(x, \la)) B_{1j}(\rho), \\ q = \overline{1, m}.
\end{multline}          
Taking $x = 0$ and using \eqref{asymptS1}, \eqref{asymptS2}, \eqref{defc0}, \eqref{defce} and \eqref{asymptF}, we obtain
\begin{equation} \label{smF2}
 	F_{1q}(0, \rho) = \rho^{-\mu_{1q}} c_{10q} e_q^* (B_{11}^0(\rho) - B_{11}(\rho)) = O(\rho^{-\beta}), \quad |\rho| \to \iy.
\end{equation}
This relation implies \eqref{asymptB} for $k = j = 1$. Using \eqref{defF2} and \eqref{smF2}, we get
$$
  	F_{2q}(x, \rho) = \Bigl( F_{1q}(x, \rho) - e_q^* (\rho x)^{-\mu_{1q}} C_1(x, \la) (B_{11}^0(\rho) - B_{11}(\rho)) \Bigr) (\rho x)^{\mu_{1q} - \mu_{2q}}.
$$
Taking $x = 0$ and using \eqref{smF1}, \eqref{asymptS1}, \eqref{asymptS2} and \eqref{asymptF}, we derive the estimate
$$
  	F_{2q}(0, \rho) = \rho^{-\mu_{2q}} c_{2q0} e_q^* (B_{12}^0(\rho) - B_{12}(\rho)) = O(\rho^{-\beta}), \quad |\rho| \to \iy.
$$
Thus, the relation \eqref{asymptB} holds for $k = 1$, $j = 2$. The proof for $k = 2$ is analogous.
\end{proof}

\begin{lem} \label{lem:asymptS}
The following asymptotic formulas hold for $\nu = 0, 1$, $x \in (0, T]$, $\mbox{Re}\, \rho \ge 0$, $|\rho| \to \iy$:
$$
   S_1^{(\nu)}(x, \la) = \frac{i}{2} \bigl( (i \rho)^{\nu} \exp(i \rho x) B_{22}^0 [I]_{\beta} - (-i \rho)^{\nu} \exp(- i \rho x) B_{12}^0 [I]_{\beta} \bigr) (D_1(\rho))^{-1},
$$
$$
   S_2^{(\nu)}(x, \la) = \frac{i}{2} \bigl( -(i\rho)^{\nu} \exp(i \rho x) B_{21}^0 [I]_{\beta} + (-i\rho)^{\nu} \exp(- i \rho x) B_{11}^0 [I]_{\beta} \bigr) (D_2(\rho))^{-1}. 
$$
\end{lem}

\begin{proof}
Using \eqref{YSB} and \eqref{asymptB}, one can easily derive the relations
$$
 	S_1(x, \la) = Y_1(x, \rho) B_{22}^*(\rho) - Y_2(x, \rho) B_{12}^*(\rho), 
$$
$$
  	S_2(x, \la) = -Y_1(x, \rho) B_{21}^*(\rho) + Y_2(x, \rho) B_{11}^*(\rho),
$$
where
$$
 	B^*_{kj}(\rho) = B_{kj}^0 [I]_{\beta} (D_j(\rho))^{-1}, \quad j, k = 1, 2.
$$
Using \eqref{asymptY}, we immediately arrive at the claim of the lemma.
\end{proof}

\bigskip

{\large \bf 5. Spectral data}

\bigskip

In this section, we introduce boundary conditions for equation \eqref{eqv} and the
notion of spectral data for the boundary value problem.

Introduce the linear forms 
$\sigma_1(Y) := -\langle S_2^*(x, \la), Y \rangle$, 
$\sigma_2(Y) := \langle S_1^*(x, \la), Y \rangle$.
In view of \eqref{wronS}, we have $\sigma_j(S_k) = \de_{jk} I$.
Note that for the classical matrix Sturm-Liouville equation
we have $\sigma_1(Y) = Y(0)$, $\sigma_2(Y) = Y'(0)$.

Consider the boundary value problem $L = L(Q, h, H)$ for equation \eqref{eqv} with the boundary conditions
\begin{gather} \label{BCU}
 	U(Y) := \sigma_2(Y) - h \sigma_1(Y) = 0, \\
 	\nonumber
   V(Y) := Y'(T) + H Y(T) = 0,
\end{gather}
where $h$ and $H$ are $m \times m$ matrices.
One can also take the Dirichlet-type boundary condition $\sigma_1(Y) = 0$ at $x = 0$.
If $\nu_m \ge \frac{1}{2}$, it is equivalent to the standard Dirichlet boundary condition $Y(0) = 0$.
Similarly, one can investigate the matrix Sturm-Liouville equation with Bessel-type singularities at the both ends of the interval.
Then both boundary conditions take the form similar to \eqref{BCU}.
 
Let $\Phi(x, \la)$ be the matrix solution of equation \eqref{eqv}, satisfying the conditions $U(\Phi) = I$, $V(\Phi) = 0$.
The matrix function $M(\lambda) := \sigma_1(\Phi(x, \lambda))$. 
is called the {\it Weyl matrix} of the problem $L$.
The Weyl matrix is a natural spectral characteristic for matrix Sturm-Liouville operators (see \cite{Bond11, Bond15-2}).
It generalizes the notion of the Weyl function for the scalar case $m = 1$ (see \cite{Mar77, FY01}).

Let $\vv(x, \la)$ be the matrix solution of equation \eqref{eqv} under initial conditions 
$\sigma_1(\vv) = I$, $\sigma_2(\vv) = h$. Obviously, 
\begin{equation} \label{relphi}
	\vv(x, \la) = S_1(x, \la) + S_2(x, \la) h,
\end{equation}
\begin{equation} \label{relPhi}
\Phi(x, \la) = S_2(x, \la) + \vv(x, \la) M(\la), \quad M(\la) = -(V(\vv))^{-1} V(S_2).
\end{equation}

The eigenvalues of the boundary value problem $L$ coincide with the zeros of the characteristic function
$\Delta(\la) := V(\vv(x, \la))$. We denote the eigenvalues, counting with their multiplicities,
by $\{ \la_p \}_{p \ge 1}$, $|\la_p| \le |\la_{p + 1}|$.
The Weyl matrix $M(\la)$ is meromorphic in $\la$ and its poles coincide 
with the eigenvalues of $L$. Assume that all the poles of $M(\la)$ are simple.
Note that in this case the multiplicities of the eigenvalues (the number of the corresponding vector eigenfunctions)
equal to the multiplicities of the zeros of the analytic function $\Delta(\la)$
(the proof is similar to \cite[Lemma 4]{Bond11}).
We will call the residues $\al_p := \Res\limits_{\la = \la_p} M(\la)$ the {\it weight matrices}.
The collection $\{ \la_p, \al_p \}_{p \ge 1}$ is called {\it the spectral data} of the problem $L$.
One can easily show that
\begin{equation} \label{Mseries}
 	M(\la) = \sum_{p = 1}^{\iy} \frac{\al_p}{m_p(\la - \la_p)}.
\end{equation}
Here $m_p$ is the multiplicity of the eigenvalue $\la_p$.

In \cite{Bond15} the inverse problem is studied, which consists in recovering of the boundary value problem $L$
from the Weyl matrix. In view of \eqref{Mseries}, this problem is equivalent to the following one.

\medskip

{\bf Inverse problem.} Given the spectral data $\{ \la_p, \al_p \}_{p \ge 1}$, determine $Q$, $h$ and $H$.

\medskip

Moreover, we will show that the matrix $\omega$ can also be determined uniquely from the spectral data.

We plan to devote a separate paper to the formulated inverse problem.
In the next two sections of this paper, we derive asymptotic formulas for $\la_p$ and $\al_p$.
They will be useful for the solution of the inverse problem.

\bigskip

{\large \bf 6. Asymptotics of the eigenvalues}

\bigskip

In order to obtain asymptotic formulas for the eigenvalues, we need the following auxiliary result.

\begin{lem} \label{lem:poly}
Let $\{ \de_n \}_{n \ge 1}$ and $\{ \ga_n \}_{n \ge 1}$ be two sequences of nonzero numbers, such that
\begin{equation} \label{poly}
 	\de_n^k + a_{1n} \de_n^{k-1} + a_{2n} \de_n^{k - 2} + \dots + a_{k-1, n} \de_n + a_{kn} = 0, \quad n \in \mathbb{N},
\end{equation}
where
$$
 	a_{jn} = \sum_{l = 0}^j O(\ga_n^l) \cdot o(\de_n^{j - l}).
$$ 
Then $\de_n = O(\ga_n)$ as $n \to \iy$.
\end{lem}

\begin{proof}
Suppose that, on the contrary, there exists such subsequence $\{ \de_{n_p} \}_{p \ge 1}$, that
$\ga_{n_p} = o(\de_{n_p})$ as $p \to \iy$. Then $a_{n_p j} = o(\de_{n_p}^j)$ and,
in view of \eqref{poly}, $\de_{n_p}^k = o(\de_{n_p}^k)$, that is impossible. Hence $\de_n = O(\ga_n)$, $n \to \iy$. 
\end{proof}

\begin{thm} \label{thm:eigen}
The eigenvalues $\{ \la_p \}_{p \ge p_0}$, starting from some number $p_0$, can be renumbered
by two indices $(n, q)$ in such a way, that the following relations hold:
\begin{equation} \label{asymptrho}
 	\rho_{nq} := \sqrt{\la_{nq}} = \frac{\pi}{T}\Bigl(n + \left\{ \frac{\mu_{1q}}{2} \right\} + O(n^{-\beta}) \Bigl), \quad n \in \mathbb{N}, \, n \ge n_0, \, q = \overline{1, m},
\end{equation}
where $\{ . \}$ stands for the fractional part: $\{ \frac{\mu_{1q}}{2} \} \in [0, 1)$, 
$ \frac{\mu_{1q}}{2} - \{ \frac{\mu_{1q}}{2} \} \in \mathbb{Z}$.
\end{thm}

\begin{proof}
Let us start with an asymptotic formula for $\Delta(\la)$. It follows from Lemma~\ref{lem:asymptS} and \eqref{relphi}, that
$$
 	\vv^{(\nu)}(x, \la) = \frac{i}{2} \bigl( (i \rho)^{\nu} \exp(i \rho x) B_{22}^0 [I]_{\beta} - (-i \rho)^{\nu} \exp(-i \rho x) B_{12}^0 [I]_{\beta} \bigr) (D_1(\rho))^{-1},
$$
where $\nu = 0, 1$, $x \in (0, T]$, $\mbox{Re}\, \rho \ge 0$, $|\rho| \to \iy$.
Consequently,
\begin{equation} \label{asymptVphi}
 	V(\vv) = -\frac{\rho}{2} \bigl( \exp(i \rho T) B_{22}^0 [I]_{\beta} + \exp(-i \rho T) B_{12}^0 [I]_{\beta} \bigr) (D_1(\rho))^{-1},
\end{equation}
\begin{equation} \label{asymptDelta}
  	\Delta(\rho^2) = (-2)^m \rho^P \det( \exp(i \rho T) B_{22}^0 + \exp(-i \rho T) B_{12}^0 + O(\exp(|\tau| T))), \quad |\rho| \to \iy, 
\end{equation}
where $P := \sum\limits_{q = 1}^m \mu_{2q}$, $\tau = \mbox{Im}\, \rho$.
The function 
$$
   \Delta_0(\rho) := (-2)^m \rho^P \det( \exp(i \rho T) B_{22}^0 + \exp(-i \rho T) B_{12}^0)
$$
is analytic for $\mbox{Re} \rho > 0$, and using \eqref{relbeta}, one can easily find its zeros (counting with their multiplicities):
\begin{equation*}
 	\rho_{nq}^0 = \frac{\pi}{T}\Bigl(n + \left\{ \frac{\mu_{1q}}{2} \right\} \Bigr), \quad n \in \mathbb{N}, \, q = \overline{1, m}. 
\end{equation*}
It is easy to show that 
\begin{equation} \label{Deltalow}
   |\Delta_0(\rho)| \ge C_{\de} |\rho|^P \exp(m |\tau| T),
\end{equation}
$$
\rho \in G_{\de} := \{ \rho \colon \mbox{Im}\,\rho \ge 0, \, |\rho| \ge \de, \, 
   |\rho - \rho_{nq}^0| \ge \de, \, n \in \mathbb{N}, \, q = \overline{1, m} \}, \quad \de > 0.
$$
The estimates \eqref{asymptDelta} and \eqref{Deltalow} imply
$|\Delta(\rho^2) - \Delta_0(\rho)| < |\Delta_0(\rho)|$ for sufficiently large $\rho \in G_{\de}$.
Consider the contours $\ga_{nq} = \{ \rho \colon |\rho - \rho_{nq}^0| = \de \}$, where $\de$ is so small, that $\ga_{nq}$ 
with distinct centers do not intersect. By virtue of Rouche's theorem, the function $\Delta(\rho^2)$ and
$\Delta_0(\rho)$ have the same number of zeros inside the contours $\ga_{nq}$ for sufficiently large $n \ge n_0$.
Applying Rouche's theorem for the contours
$$
 	\Gamma_R := \partial S_R, \quad S_R := \{ \rho \colon \mbox{Re}\,\rho > 0, \, |\rho| > 1, \, |\rho| < R \}
$$ 
for such $R$ that $\Gamma_R \subset G_{\de}$, we show that the functions $\Delta(\rho^2)$ and $\Delta_0(\rho)$ have the
same number of zeros in $S_R$. Consequently, all zeros of $\Delta(\rho^2)$ with sufficiently large $|\rho|$ lie
inside the circles $\ga_{nq}$. Since $\de > 0$ can be chosen arbitrarily small, the square roots of the eigenvalues $\la_p$, $p \ge p_0$
admit the asymptotic representations
\begin{equation} \label{rho0}
 	\rho_{nq} = \rho_{nq}^0 + \de_{nq}, \quad n \ge n_0, \, q = \overline{1, m}, 
\end{equation}
where $\de_{nq} = o(1)$ as $n \to \iy$.

It remains to show that $\de_{nq} = O(n^{-\beta})$. Fix $q = \overline{1, m}$.
Substitute $\rho = \rho_{nq}$ in the form \eqref{rho0} into the relation $\Delta(\rho^2) = 0$.
Using the asymptotic formula \eqref{asymptDelta}, one can easily obtain for 
the sequence $\{ \de_{nq} \}_{n \ge n_0}$ the relations in the form \eqref{poly}
with $\ga_n = n^{-\beta}$. Applying Lemma~\ref{lem:poly}, we arrive at the assertion of the theorem.
\end{proof}

\begin{cor}
The number $p_0 - n_0 m$, which characterizes the ``shift'' of the spectrum with respect to the sequence $\{ (\rho_{nq}^0)^2 \}_{n \ge 0, q = \overline{1, m}}$,
does not depend on $Q(x)$, $h$ and $H$. It depends only on the matrix $\omega$.
\end{cor}

\begin{proof}
First of all, note that the number $p_0 - n_0 m$ does not depend on the choice of $n_0$.

Consider the boundary value problem $\tilde L := L(\tilde Q, \tilde h, \tilde H)$, $\tilde Q = \tilde h = \tilde H = 0$.
Its characteristic function $\tilde \Delta(\rho^2)$ satisfies the estimates \eqref{asymptDelta} and \eqref{Deltalow}
for sufficiently large $\rho \in G_{\de}$. 
 Therefore one can apply Rouche's theorem to the functions $\tilde \Delta(\la)$ and $\Delta(\la) - \tilde \Delta(\la)$ on the contour $\Theta_R := \{ \la \colon |\la| = R \}$ with sufficiently large $R$,
such that $\{ \rho \colon \rho^2 \in \Theta_R \} \subset G_{\de}$. Thus, the problems $L$ and $\tilde L$ have the same number of
the eigenvalues in the region $\mbox{int}\, \Theta_R$, and the other eigenvalues of the both problems satisfy \eqref{asymptrho}.
This fact immediately yields the claim.
\end{proof}

\begin{cor}
If all the numbers $\{ \frac{\mu_{1q}}{2} \}$, $q = \overline{1, m}$, are distinct, then the poles $\la_{nq}$ of the Weyl matrix $M(\la)$ 
are simple for sufficiently large values of $n$.
\end{cor}

\bigskip

{\bf 7. Asymptotics of the weight matrices}

\bigskip

Consider the eigenvalues $\la_{nq}$, $n \ge n_0$, $q = \overline{1, m}$, 
numbered in accordance with Theorem~\ref{thm:eigen}. Let $m_{nq}$ be their multiplicities,
and
$$
 	\al_{nq} = \Res_{\la = \la_{nq}} M(\la).
$$

Denote $J_q = \Bigl\{ s \colon \left\{ \frac{\mu_{1s}}{2} \right\} = \left\{ \frac{\mu_{1q}}{2} \right\} \Bigr\}$. 
In this section, we calculate the asymptotics of the sums $\sum\limits_{s \in J_q} m_{ns}^{-1} \al_{ns}$
as $n \to \iy$. Each of these sums corresponds to the values $\rho_{nq}$ inside one of the contours 
$\ga_{nq} = \{ \rho \colon |\rho - \rho_{nq}^0| = \de \}$, 
defined in the proof of Theorem~\ref{thm:eigen}. The matrices $\al_{ns}$ are multiplied by $m_{ns}^{-1}$
in order to count the residues at equal poles only once.

\begin{thm} \label{thm:weight}
The following relation holds for $n \ge n_0$:
\begin{equation} \label{asymptal}
\sum\limits_{s \in J_q} m_{ns}^{-1} \al_{ns} = \frac{\pi n}{T} D_1\Bigl(\frac{\pi n}{T}\Bigr) (A_q + O(n^{-\beta})) \Bigl(D_2\Bigl(\frac{\pi n}{T}\Bigr)\Bigr)^{-1}, \quad q = \overline{1, m},
\end{equation}
where $A_q = [A_{q, jk}]_{j, k = 1}^m$,
$$
 	A_{q, jk} = \begin{cases}
					\theta_j, \quad j = k \: \text{and} \: j \in J_q, \\
					0, \quad \text{otherwise},
 				\end{cases}
 	\quad \theta_j = - \frac{\beta_{11j}^0}{i T \beta_{12j}^0} (1 - \exp(- 2 \pi i \nu_j)) \ne 0.
$$
\end{thm}                  

\begin{proof}
Using Lemma~\ref{lem:asymptS}, we obtain
\begin{equation} \label{asymptVS}
 	V(S_2) = \frac{\rho}{2} \bigl( \exp(i \rho T) B_{21}^0 [I]_{\be} + \exp(-i \rho T) B_{11}^0 [I]_{\be} \bigr) (D_2(\rho))^{-1}, \quad |\rho| \to \iy.
\end{equation}
The relations \eqref{asymptrho}, \eqref{relPhi}, \eqref{asymptVphi} and \eqref{asymptVS} imply 
\begin{multline} \label{asymptM}
 	M(\rho^2) = -(V(\vv))^{-1} V(S_2) = D_1(\rho) \bigl( \exp(i \rho T) B_{22}^0 + \exp(-i \rho T) B_{12}^0 + O(\exp(|\tau|T) \rho^{-\beta}) \bigr)^{-1} \\
 	\cdot \bigl( \exp(i \rho T) B_{21}^0 + \exp(-i \rho T) B_{11}^0 + O(\exp(|\tau|T) \rho^{-\beta})\bigr) (D_2(\rho))^{-1}, \quad |\rho| \to \iy.
\end{multline}
The following function
$$
 	M_0(\rho) = D_1(\rho) \bigl( \exp(i \rho T) B_{22}^0 + \exp(-i \rho T) B_{12}^0 ) \bigr)^{-1} 
 	\bigl( \exp(i \rho T) B_{21}^0 + \exp(-i \rho T) B_{11}^0 \bigr) (D_2(\rho))^{-1}.
$$
is meromorphic for $\mbox{Re}\, \rho > 0$.

Fix $q = \overline{1, m}$. By virtue of the residue theorem 
\begin{equation} \label{smM1}
   \sum\limits_{s \in J_q} m_{ns}^{-1} \al_{ns} = \sum\limits_{s \in J_q}  m_{ns}^{-1} \Res_{\rho = \rho_{ns}} (2 \rho M(\rho^2)) = 
   \frac{1}{2 \pi i} \int_{\ga_{nq}} (2 \rho M(\rho^2)) \, d \rho
\end{equation}
for sufficiently large $n$. 
It follows from \eqref{asymptrho} and \eqref{asymptM}, that
\begin{equation} \label{smM2}
   \frac{1}{2 \pi i} \int_{\ga_{nq}}(2 \rho M(\rho^2)) \, d \rho - \frac{1}{2 \pi i} \int_{\ga_{nq}} (2 \rho M_0(\rho)) \, d \rho = D_1\Bigl(\frac{\pi n}{T}\Bigr) O(n^{1-\beta}) \Bigl(D_2\Bigl(\frac{\pi n}{T}\Bigr)\Bigr)^{-1}.	
\end{equation}
The integral of $M_0(\rho)$ can be calculated by the residue theorem. Note that 
the matrix function $M_0(\rho)$ is diagonal. The $s$-th diagonal element $M_{0, ss}(\rho)$ has the only simple pole
$\rho_{nq}^0$ inside the contour $\ga_{nq}$, if $s \in J_q$, and it is analytic otherwise.
Using \eqref{asymptrho} and \eqref{relbeta}, we calculate
\begin{multline*}
 	\Res_{\rho = \rho_{nq}^0} (2 \rho M_{0, ss}(\rho)) = \lim_{\rho \to \rho_{nq}^0} (\rho - \rho_{nq}^0) \cdot 2 \rho^{1 + \mu_{s1} - \mu_{s2}} 
 	\frac{\exp(i \rho T) \beta_{21s}^0 + \exp(-i \rho T) \beta_{11s}^0 }{\exp(i \rho T) \beta_{22s}^0 + \exp(-i \rho T) \beta_{12s}^0} \\
 	= 2 \Bigl(\frac{\pi n}{T}\Bigr)^{1 -2 \nu_s} (1 + O(n^{-1})) \frac{\beta_{11s}^0 (\exp(i \rho_{nq}^0 T) \exp(i \pi \mu_{1s}) + \exp(-i \rho_{nq}^0 T))}{i T \beta_{12s}^0 
 	(\exp(i \rho_{nq}^0 T) \exp(i \pi \mu_{2s}) - \exp(-i \rho_{nq}^0 T))}
 	\\ = \theta_s \Bigl( \frac{\pi n}{T}\Bigr)^{1 - 2 \nu_s} (1 + O(n^{-1})).
\end{multline*}  
Combining this result with \eqref{smM1} and \eqref{smM2}, we arrive at \eqref{asymptal}.
\end{proof}

\begin{cor}
The matrix $\om$ is uniquely determined by the weight matrices $\{ \al_p \}_{p \ge 1}$ or by the Weyl matrix $M(\la)$.
\end{cor}

Indeed, the numbers $\nu_q$, $q = \overline{1, m}$, can be found from the diagonal elements 
of the sums $\sum\limits_{s \in J_q} m_{ns}^{-1} \al_{ns}$ or the integral
$\frac{1}{2 \pi i} \int\limits_{\ga_{nq}} (2 \rho M(\rho)) \, d\rho$.

{\bf Acknowledgments}. This work was supported by Grant 1.1436.2014K
of the Russian Ministry of Education and Science and by Grants 15-01-04864 and 16-01-00015
of Russian Foundation for Basic Research.

\medskip

\noindent Natalia Bondarenko \\
Department of Mechanics and Mathematics, Saratov State University, \\
Astrakhanskaya 83, Saratov 410012, Russia, \\
e-mail: {\it BondarenkoNP@info.sgu.ru}


\medskip

\begin{thebibliography}{99}

\bibitem{FY05-1}
Freiling G., Yurko V. Boundary value problems with regular singularities and singular boundary conditions.
International Journal of Mathematics and Mathematical Sciences 2005:9 (2005) 1481--1495.

\bibitem{FY05-2}
Freiling G., Yurko V. Inverse problems for differential operators with singular boundary conditions.
Math. Nachr. 278, No. 12�--13 (2005), 1561--1578.

\bibitem{FY99}
Freiling, G.; Yurko, V. A. Reconstructing parameters of a medium from incomplete spectral information. Results
Math. 35 (1999), 228--249.

\bibitem{And97} 
Anderssen, R. S. The effect of discontinuities in density and shear velocity on the asymptotic overtone structure of
tortional eigenfrequencies of the Earth. Geophys. J. R. Astr. Soc. 50 (1997), 303--309.

\bibitem{LU81} Lapwood, F. R; Usami, T. 
Free oscillations of the Earth. Cambridge: Cambridge University Press, 1981.

\bibitem{Yur97}
Yurko, V. A. On integral transforms connected with differential operators having singularities
inside the interval Integral Transforms Spec. Functions 5 (1997), 309--322.

\bibitem{Fed13}
Fedoseev, A. E. An inverse problem for Sturm-Liouville operators on the half-line having Bessel-type singularity
in an interior point, Cent. Eur. J. Math., 11(12) (2013), 2203--2214.

\bibitem{GYS15}
Gorbunov, O.; Shieh, C.-T.; Yurko, V. Dirac system with a singularity in an interior point,
Applicable Analysis (2015), DOI: 10.1080/00036811.2015.1091069.

\bibitem{GY15}
Gorbunov, O.; Yurko, V. Inverse problem for dirac system with singularities
in interior points, Anal. Math. Phys.(2015), 1--29, DOI 10.1007/s13324-015-0097-1.

\bibitem{Ign15}
Ignatiev, M. Inverse scattering problem for Sturm-Liouville operators with Bessel singularities
on noncompact star-type graphs, Inverse Problems 31 (2015) 125006 (14pp).

\bibitem{Yur15}
Yurko, V. Inverse problems for higher order differential systems with regular singularities on star-type graphs,
Tamkang J. Math. 46:3 (2015), 257--268.

\bibitem{Yur-depon}
Yurko, V. A. On differential operators of high orders with a singularity [In Russian]. Dep. at VINITI 20.06.90, No. 3530-90, Saratov (1990), 19 p.

\bibitem{Yur92}
Yurko, V. A. An inverse problem for differential equations with a singularity, Differ. Uravn. 28
(1992), no. 8, 1355--1362 (Russian), English translation in Differ. Equ. 28 (1992), no. 8,
1100–1107.

\bibitem{FY01}
Freiling, G.; Yurko, V. Inverse Sturm-Liouville problems and their applications. Huntington,
NY: Nova Science Publishers, 305 p. (2001).

\bibitem{AM60}
Agranovich, Z. S.; Marchenko, V. A. The inverse problem of scattering theory [in Russian],
KSU, Kharkov, 1960;
Gordon and Breach, New York, 1963 (Eng. Transl.).

\bibitem{Bond15}
Bondarenko, N. An Inverse Spectral Problem for the Matrix Sturm-Liouville Operator with a Bessel-Type Singularity. International Journal of Differential Equations.
Vol. 2015 (2015), Article ID  647396. 4 p.

\bibitem{Bond11}
Bondarenko, N. Spectral analysis for the matrix Sturm-Liouville operator on a finite interval,
Tamkang J. Math. 42:3 (2011), 305--327.

\bibitem{Bond15-2}
Bondarenko, N. Recovery of the matrix quadratic differential pencil from the spectral data,
J. Inverse Ill-Posed Probl. (2015), DOI 10.1515/jiip-2014-0074.

\bibitem{Mar77}
Marchenko, V. A. Sturm-Liouville Operators and their Applications, Naukova Dumka,
Kiev (1977) (Russian); English transl., Birkhauser (1986).

\end{thebibliography}
\end{document}